\documentclass[12pt]{amsart}

\numberwithin{equation}{section}

\usepackage{amsmath,amsthm,amsfonts,eucal,}
\usepackage{graphicx}
\usepackage{xypic}
\usepackage{amssymb}

\def\cb{{\mathcal B}}

\def\cd{{\mathcal D}}

\def\cf{{\mathcal F}}
\def\cg{{\mathcal G}}
\def\ch{{\mathcal H}}

\def\cam{{\mathcal M}}

\def\cu{{\mathcal U}}

\def\ga{{\mathfrak A}}

\def\gf{{\mathfrak F}}
\def\gpg{{\mathfrak g}}

\def\bc{{\mathbb C}}

\def\bn{{\mathbb N}}

\def\br{{\mathbb R}}

\def\bz{{\mathbb Z}}

\def\a{\alpha}
\def\b{\beta}
\def\g{\gamma}  
\def\d{\delta}  

\def\eps{\varepsilon}

\def\m{\mu}
\def\p{\pi}

\def\r{\rho}
\def\s{\sigma} 

\def\f{\varphi}  
  
\def\om{\omega} \def\Om{\Omega}

\newtheorem{thm}{Theorem}[section]
\newtheorem{lem}[thm]{Lemma}
\newtheorem{cor}[thm]{Corollary}
\newtheorem{prop}[thm]{Proposition}

\theoremstyle{definition}

\newtheorem{defin}[thm]{Definition}

\def\dim{\mathop{\rm dim}}
\def\min{\mathop{\rm min}}
\def\supp{\mathop{\rm supp}}
\def\Im{\mathop{\rm Im}}

\def\sgn{\mathop{\rm sgn}}

\def\sp{\mathop{\rm span}}

\begin{document}

%\linenumbers

\title[Weakly Monotone Fock Space]
{Weakly Monotone Fock Space and monotone convolution of the Wigner law}
\author{Vitonofrio Crismale}
\address{Vitonofrio Crismale\\
Dipartimento di Matematica\\
Universit\`{a} degli studi di Bari\\
Via E. Orabona, 4, 70125 Bari, Italy}
\email{\texttt{vitonofrio.crismale@uniba.it}}
\author{Maria Elena Griseta}
\address{Maria Elena Griseta\\
Dipartimento di Matematica\\
Universit\`{a} degli studi di Bari\\
Via E. Orabona, 4, 70125 Bari, Italy}
\email{\texttt{maria.griseta@uniba.it}}
\author{Janusz Wysocza\'nski}
\address{Janusz Wysocza\'nski\\
Institute of Mathematics\\
Wroclaw University\\
pl. Grunwaldzki 2/4, 50-384 Wroclaw, Poland}
\email{{\texttt{jwys@math.uni.wroc.pl}}}

\date{\today}

\begin{abstract}
We study the distribution (w.r.t. the vacuum state) of family of partial sums $S_m$ of position operators on weakly monotone Fock space. We show that any single operator has the Wigner law, and an arbitrary family of them (with the index set linearly ordered) is a collection of monotone independent random variables. It turns out that our problem equivalently consists  in finding the $m$-fold monotone convolution of the semicircle law. For $m=2$ we compute the explicit distribution. For any $m>2$ we give the moments of the measure, and show it is absolutely continuous and compactly supported on a symmetric interval whose endpoints can be found by a recurrence relation.
\end{abstract}

\maketitle

\section{introduction}
\label{sec1}

Weakly monotone Fock space was first introduced in \cite{Wys2005}, mainly in order to exhibit the first construction of monotone independent non-commutative random variables with the \emph{arcsine} law, and different from the gaussian operators in monotone Fock spaces \cite{[Lu95a],Mur}. It belongs to the collection of Fock spaces obtained by the so-called Yang-Baxter-Hecke quantization \cite{Boz}, which also contains Bose, Fermi, boolean and monotone Fock spaces, as well as the so-called $Q$-deformed Fock spaces, investigated in \cite{BLW} for their natural applications to the study of L\'{e}vy processes for anyon statistics. The strong ergodic properties for $C^*$-dynamical systems arising from Yang-Baxter-Hecke deformation of usual Fock spaces were studied in \cite{erg}, whereas in \cite{CFG} the states in monotone $*$-algebra invariant under some distributional symmetries were classified.

In this paper we study the distributions (with respect to the vacuum state) of family of partial sums $\displaystyle S_m:=\sum_{i=1}^{m} G_i$ of position operators $G_i:=A_i+A_i^{\dag}$ on the weakly monotone Fock space $\gf_{WM}(\ch)$, $\ch$ being a separable Hilbert space. Here, $A_i$ and $A^\dag_i$ are the annihilation and creation operators with the test function given by any arbitrary element of the canonical basis of $\ch$, respectively. We establish the $G_i$ are monotone independent, and moreover any of them has the distribution given by the (absolutely continuous) Wigner semicircle law with density $\nu(dx):=\displaystyle \frac{1}{2\pi}\sqrt{4-x^2}$ on $[-2,2]$. Up to our knowledge, this is the first example of a family of monotonically independent non commutative random variables with the Wigner law. One further notices a deep difference with respect to the case of monotone Fock space, where the analogous operators are Bernoulli distributed onto the two points $-1$ and $1$. The previous mentioned results here obtained also suggest that the distribution of $S_m$ is given by the $m-$fold monotone convolution \cite{Mu} of the Wigner measure with itself $\nu^{\rhd m}:=\nu\rhd \nu\rhd \cdots \rhd \nu$. As a consequence, our investigation here can be also viewed as a study of the monotone convolution of the semicircle law.

More in detail, in Section \ref{sec2} we first recall the definitions of weakly monotone Fock space and the basic operators on it. Then, in Theorem \ref{monotone independence thm}, we show that the $*$-algebras generated by any single $A_i$ are monotonically independent among the bounded operators on the weakly monotone Fock space. This result entails automatically the monotone independence of position operators. The section ends by a quick review of the Cauchy transform of a measure and monotone convolution.

Section \ref{sec3}, organized into two subsections, opens with the proof that any $G_i$ has the semicircle law as its vacuum distribution. In the first subsection we deal with the problem of sum of two position operators. The law, whose moments are computed in Proposition \ref{momcase2}, is further explicitly found in Theorem \ref{dens} as an absolutely continuous measure supported in $[-\frac{5}{2},\frac{5}{2}]$. Here we stress such results are obtained without using monotone independence. On the contrary, an explicit computation of the distribution for a sum of at least three operators appears quite complicated, even exploiting the monotone convolution. Nevertheless this feature, as well as the properties of Cauchy transform and its reciprocal map, allows us to state that such laws, corresponding to $\nu^{\rhd m}$, $m$ being the number of operators in the sum, are absolutely continuous, symmetric and compactly supported on intervals of the form $[-a_m, a_m]$. This is the main result of the second subsection. Furthermore, a nice recurrence relation for the right endpoints of the above intervals is achieved, namely $a_1=2$ and  $\displaystyle a_{m+1}=a_m+\frac{1}{a_m}$ for any $m$. This entails that the sequence $\big(\frac{a_m}{\sqrt{m}}\big)_m$ is decreasing and converges to $\sqrt{2}$. Of course this should be expected since, by the monotone CLT, $\displaystyle \frac{1}{\sqrt{m}}S_m$ weakly converges to the arcsine law with density $\displaystyle \frac{1}{\pi\sqrt{2-x^2}}dx$ on $(-\sqrt{2}, \sqrt{2})$. We point out that the above presented results apply to any sum $\displaystyle\sum_{k\in I}G_k$, where $I$ is a finite set with $m$ elements. Equivalent study of moments defines a family of positive definite sequences $\displaystyle \{\left(d_n^{(m)} \right)_{n\geq 0}\mid m=1, 2, \ldots \}$, where $\displaystyle d_n^{(m)} := \om_{\Om}\left(\left(S_m\right)^{2n}\right)$. Here we consider only the even moments, since the odd ones vanish for each $m$. We show that this family satisfies the recurrence
\begin{equation*}
d_n^{(0)}\equiv 1, \ d_n^{(1)}=C_n, \quad \text{and} \quad d_n^{(m)}=\sum_{k=1}^{n}d_{n-k}^{(m)}\sum_{j=1}^{m}d_{k-1}^{(j)},
\end{equation*}
where $C_n$ are the Catalan numbers. As a consequence each sequence consists of positive integers, which can be regarded as the $m-$fold monotone convolutions of the sequence of Catalan numbers, namely $\displaystyle \left(d_n^{(m)}\right)_{n\geq 0}:=\left(\left(C_n \right)_{n\geq 0}\right)^{\rhd m}$. To get a flavour, here we list few examples of them, the reader being addressed to the Appendix for more information.
\begin{eqnarray*}
d_n^{(2)}&=&1, 2, 7, 29, 131, 625, 3099, 15818, 82595, 439259, \ldots \\
d_n^{(3)}&=&1, 3, 15, 87, 544, 3566, 24165, 167904,  8568923, \ldots \\
d_n^{(4)}&=&1, 4, 26, 194, 1551, 12944, 111313, 979009, 8764089,  \ldots
\end{eqnarray*}
In the appendix one also finds the computation of some terms of the sequence of monotone cumulants \cite{HasSa} of the Wigner law, as well some of the orthogonal polynomials for $m=2$.

Finally, we show that $\displaystyle \left(d_n^{(m)}\right)_{m\geq 0}$ with fixed $n$, are indeed sequences of polynomials in $m$ of degree $n$. Examples of the first of these polynomials are
\begin{equation*}
d_0^{(m)}\equiv 1, \quad d_1^{(m)} = m, \quad d_2^{(m)} = \frac{3m^2+m}{2}, \quad d_3^{(m)} = \frac{m(5m^2+4m+1)}{2}\,.
\end{equation*}

\section{preliminaries}
\label{sec2}

This section is mainly devoted to recall some notions and features and show some new results useful throughout the paper. We start with the so-called weakly monotone Fock space and the properties of creation and annihilation operators defined on it, the interested reader being addressed to \cite{Wys2005} for more details.

\subsection{Weakly monotone Fock space}
\label{sec2a}
Let $\ch$ be a separable Hilbert space with a fixed orthonormal basis $\{e_i:i\geq1\}$. By $\gf(\ch)$ we denote the full Fock space on $\ch$, the vacuum vector is $\Om=1\oplus0\oplus\ldots$, and $a_i:=a(e_i)$ and $a_i^{\dag}:=a(e_i)^{*}$ are the standard annihilation and creation operators by the
vector $e_i$, respectively. The \emph{weakly monotone Fock space}, in the sequel denoted by $\gf_{WM}(\ch)$, is the closed subspace of $\gf(\ch)$ spanned by $\Om$, $\ch$ and all the simple tensors of the form $e_{i_k}\otimes e_{i_{k-1}}\otimes \cdots\otimes e_{i_1}$, where $i_k\geq i_{k-1}\geq\ldots\geq i_1$
and $k\geq2$.

If the Hilbert space $\ch$ is finite dimensional with $n =\dim(\ch)\geq 2$, then the basis for $\gf_{WM}(\ch)$ consists of the vacuum and all the simple tensors
\begin{equation}
\label{basis}
e_n^{k_n}\otimes e_{n-1}^{k_{n-1}}\otimes\cdots\otimes e_1^{k_1}
\end{equation}
where $k_n, k_{n-1},\ldots,k_1\geq 0$, $e_m^{k}:=\underbrace{e_m\otimes\cdots\otimes e_m}_k$ if $k\geq1$, and the convention that $e^{k_i}_{i}$ does not appear in \eqref{basis} if $k_i=0$.

The weakly monotone creation and annihilation operators with "test function" $e_i$, denoted by $A^\dag_i$ and $A_i$ respectively, are defined as follows
$$
A_i(e_j)=\delta_{ij}\Om, \qquad A_i(e_{i_k}\otimes e_{i_{k-1}}\otimes \cdots\otimes e_{i_1})=\delta_{ii_k}e_{i_{k-1}}\otimes \cdots\otimes e_{i_1}\,,
$$
where $\delta_{ij}$ is the Kronecker symbol and
\begin{equation*}
\begin{split}
A^{\dag}_i(\Om)&=e_i \\
A^{\dag}_i(e_{i_k}\otimes e_{i_{k-1}}\otimes \cdots\otimes e_{i_1})&=\begin{cases}
e_i\otimes e_{i_k}\otimes e_{i_{k-1}}\otimes \cdots\otimes e_{i_1}&\text{if $i\geq i_k$,}\\
0&\text{if $i<i_k$.}\\
\end{cases}
\end{split}
\end{equation*}
One notices that, after denoting by $P_M$ the orthogonal projection from the full Fock space to the weakly monotone one, then $A_i=P_Ma_iP_M$ and $A^\dag_i=P_Ma^{\dag}_iP_M$. These operators are moreover adjoint to each other, and bounded with norm one. Further they satisfy the following identities
\begin{equation}
\label{cr}
\begin{split}
A^{\dag}_iA^{\dag}_j&=A_jA_i =0 \quad\text{if $i<j$,}\\
A_iA^{\dag}_j& =0 \,\,\,\,\,\,\,\,\,\,\,\,\,\,\,\,\,\,\,\,\,\,\,\,\,\,\, \text{if $i\neq j$.} \\
\end{split}
\end{equation}
For each $i\in\mathbb{N}$ we define the self-adjoint field position operators $G_i:=A_i+A_i^{\dag}$.\\
The following technical result, which can be given as the weakly monotone version of Lemma 5.4 in \cite{erg}, will be very useful through the paper.
\begin{lem}
\label{3a}
For any $k,j\geq1$, one has
\begin{align}
\label{alpha}
A_kA_jA_j^{\dag}&=\a_{j,k}A_k  & A_jA_j^{\dag}A_k^{\dag}=\a_{j,k}A_k^{\dag}\,,
\end{align}
where
\begin{equation*}
\a_{j,k}=\begin{cases}
1 & \text{if $j\geq k$,}\\
0 & \text{otherwise.}
\end{cases}
\end{equation*}
Moreover, for $j\geq k$
\begin{align}
\label{jgeqk}
A_jA^{\dag}_jA_k&=A_k  & A_k^{\dag}A_jA_j^{\dag}=A_k^{\dag}\,.
\end{align}
\end{lem}
\begin{proof}
\eqref{alpha} easily follows from the definition of creation and annihilation operators, the first and the second equality in \eqref{cr}.

For the first equality in \eqref{jgeqk}, one notices
\begin{equation*}
A_jA_j^{\dag}A_k\Om=0=A_k\Om
\end{equation*}
and further, for $k_1,\ldots, k_n>0$, $i_1<i_2<\cdots < i_n$,
\begin{equation*}
A_jA_j^{\dag}A_ke_{i_n}^{k_n}\otimes e_{i_{n-1}}^{k_{n-1}}\otimes\cdots\otimes e_{i_1}^{k_1}=\left\{
\begin{array}{ll}
 \d_{k,n}e_{i_{n-1}}^{k_{n-1}}\otimes\cdots\otimes e_{i_1}^{k_1}  & \text{if $k_n=1$,} \\
 \d_{k,n}e_{i_n}^{k_n-1}\otimes e_{i_{n-1}}^{k_{n-1}}\otimes\cdots\otimes e_{i_1}^{k_1}  & \text{if $k_n>1$.}
\end{array}
\right.
\end{equation*}
In both the cases the quantity on the r.h.s. above is nothing else than
$$
A_ke_{i_n}^{k_n}\otimes e_{i_{n-1}}^{k_{n-1}}\otimes\cdots\otimes e_{i_1}^{k_1}\,.
$$
The second identity in \eqref{jgeqk} is achieved by taking the adjoint.
\end{proof}
Using the previous Lemma is easy to show that for each $i\geq1$, $A_iA_i^{\dag}$ and $A_i^{\dag}A_i$ are self-adjoint projections, and
$$
A_iA_i^{\dag}=A_i^{m}(A_i^{\dag})^m,\,\,\,\,\,\,\,\,\,\,\,\, m\in\mathbb{N}\,.
$$

%%%%%%%%%%%%%%%%%%%%%%%%%%%%%%%%%%%%%%%%%%%%%%%%%%%%%%%%%%%%%%%%%%%%%%%%%%%%%%%%%%
\subsection{Monotone independence of position operators}
For $i\geq 1$ let $\cb_i$ be the *-algebra generated by $A_i, A_i^{\dag}$ and let $\cb$ be the unital $*-$algebra of all bounded operators on $\gf_{WM}(\ch)$. We shall show that the algebras $\{\cb_i:\ i\geq 1\}$ are monotone independent in the non-commutative probability space $(\cb, \om_{\Om})$, where we denote $\om_{\Om}(\cdot):=\langle\Om,\cdot\Om\rangle$ the vacuum state. This implies the position operators $(G_i)_{i\geq 1}$ are monotone independent.

Observe first that the closed subspace
$$
\gf_{WM}^{\leq i}(\ch):=\overline{\sp\{e_{l_m}\otimes \cdots \otimes e_{l_1}:\ l_1\leq \ldots \leq l_m \leq i, \ m\geq 1\}}
$$
is left invariant by both $A_i, A_i^{\dag}$, and so it is the space
$$
\gf_{WM}^{> i}(\ch):=\overline{\sp\{e_{l_m}\otimes \cdots \otimes e_{l_1}:\ l_1\leq \ldots \leq l_m >i, \ m\geq 1 \}},
$$
which is in the kernel of both operators. It follows in particular that $I\notin \cb_i$.

Moreover, $A_i^{\dag}A_i$ is the orthogonal projection onto
$$
\gf_{WM}^{= i}(\ch):=\overline{\sp\{e_{l_m}\otimes \cdots \otimes e_{l_1}:\ l_1\leq \ldots \leq l_m = i, \ m\geq 1\}}.
$$
To describe elements of $\cb_i$ we shall use some additional notations. $\cb_i$ is given by non-commutative polynomials in $A_i, A_i^{\dag}$,  so that $p\in \cb_i$ can be written as a finite sum $\displaystyle p=\sum_{|\a|\geq 1}c_{\a}A_i^{\a}$, where for each multi-index $\a:=(k_1, \ldots , k_n)\in \bz^n$, and $n\in \mathbb{N}$, one has $A_i^{\a}:=A_i^{k_1}\cdots A_i^{k_n}$, with $k_jk_{j+1}<0$ (alternating signs) for $1\leq j \leq n-1$.
\begin{thm}
\label{monotone independence thm}
The algebras $\{\cb_i:\ i\geq 1\}$ are monotone independent in $(\cb, \om_{\Om})$ so that they satisfy the following two conditions:

\noindent (M1) if $i<j>k$ then
\begin{equation}
\label{M1}
p_ip_jp_k=\om_{\Om}(p_j)p_ip_k,
\end{equation}
whenever $p_i\in \cb_i$, $p_j\in \cb_j$, $p_k\in \cb_k$.

\noindent (M2) if $j_1>\cdots >j_k<\cdots < j_n$ then
\begin{equation}
\label{M2}
\om_{\Om}(p_{j_1}\cdots p_{j_k} \cdots p_{j_n})= \prod_{m=1}^{n}\om_{\Om}(p_{j_m}),
\end{equation}
whenever $p_{j_m}\in \cb_{j_m}$ for $m=1, \ldots , n$.
\end{thm}
\begin{proof}
To prove (M1) we first consider an arbitrary simple tensor $v:=e_{l_m}\otimes \cdots \otimes e_{l_1}$ with $l_1\leq \ldots \leq l_m$, $m\geq 1$
and show that for $i<j>k$,
\begin{equation}
\label{M1v}
p_ip_jp_kv=\om_{\Om}(p_j)p_ip_kv.
\end{equation}
For this purpose we consider three cases.

\bigskip

\noindent (1) $k<l_m$. Here, both sides of \eqref{M1v} are zero.

\bigskip

\noindent (2) $k>l_m$. Take an arbitrary $s\geq 1$ and use the notation $p_k \Om=\om_{\Om}(p_s)\Om+h_s$, where $h_s\in \sp\{e_s^{\otimes m} \mid m\geq 1\}$.
Since $p_ih_j=0$, one has
\begin{eqnarray*}
p_ip_jp_k v &= & p_ip_j[\om_{\Om}(p_k)v+h_k\otimes v] \\
&=& p_i\om_{\Om}(p_k)[\om_{\Om}(p_j)v+h_j\otimes v]+p_i[\om_{\Om}(p_j)h_k\otimes v+ h_j\otimes h_k\otimes v] \\
&=& \om_{\Om}(p_j)p_i[\om_{\Om}(p_k)v+h_k\otimes v] \\
&=& \om_{\Om}(p_j)p_ip_k v.
\end{eqnarray*}

\bigskip

\noindent (3) $k=l_m$. Here, if $p_k v$ is not null, since $p_k v=u+c\Om$ for $u\in\gf_{WM}^{\leq k}(\ch)$ and $c\in\mathbb{C}$, $p_i h_j=0$ gives
\begin{equation*}
p_ip_jp_k v = p_i[\om_{\Om}(p_j)(u+c\Om)+h_j\otimes (u+c\Om)] = \om_{\Om}(p_j)p_ip_k v.
\end{equation*}
Hence in this case \eqref{M1v} holds.\\
Finally, when $v=\Om$,
\begin{eqnarray*}
p_ip_jp_k \Om &= & p_ip_j[\om_{\Om}(p_k)\Om+h_k] \\
&=& p_i\om_{\Om}(p_k)[\om_{\Om}(p_j)\Om+h_j]+p_i[\om_{\Om}(p_j)h_k+ h_j\otimes h_k] \\
&=& \om_{\Om}(p_j) [ \om_{\Om}(p_k) (\om_{\Om}(p_i)\Om+h_i) + p_ih_k] \\
&=& \om_{\Om}(p_j)[\om_{\Om}(p_k)p_i\Om+p_ih_k] \\
&=& \om_{\Om}(p_j)p_ip_k\Om,
\end{eqnarray*}
since again $p_ih_j=0$. \eqref{M1} then follows as the collection of simple tensors used above $v$ is linearly dense in $\gf_{WM}(\ch)$.

To prove (M2), we first notice
\begin{equation}
\label{pkpr}
p_kp_r\Om=\om_{\Om}(p_r)p_k\Om,\,\,\,\,\,\,\,\, \text{if $k<r$.}
\end{equation}
Then, for $j_1>\cdots >j_k<\cdots < j_n$, from \eqref{pkpr} we get
\begin{equation*}
\begin{split}
\om_{\Om}(p_{j_1}\cdots p_{j_k} \cdots p_{j_n}) &= \langle\Om, p_{j_1}\cdots p_{j_k} \cdots p_{j_n}\Om\rangle\\
&=\om_{\Om}(p_{j_{k+1}})\cdots \om_{\Om}(p_{j_{n}})\om_{\Om}(p_{j_1}\cdots p_{j_k}) \\
&= \om_{\Om}(p_{j_{k+1}})\cdots \om_{\Om}(p_{j_{n}})\overline{\om_{\Om}(p^{*}_{j_k}\cdots p^{*}_{j_1})} \\
&= \om_{\Om}(p_{j_{k+1}})\cdots \om_{\Om}(p_{j_{n}})\overline{\om_{\Om}(p^{*}_{j_1})}\cdots \overline{\om_{\Om}(p^{*}_{j_k})}\\
&= \om_{\Om}(p_{j_{k+1}})\cdots \om_{\Om}(p_{j_n})\om_{\Om}(p_{j_1})\cdots \om_{\Om}(p_{j_k}),\\
\end{split}
\end{equation*}
which gives \eqref{M2}.
\end{proof}
Since for the position operators we have $G_i\in \cb_i$, we achieve the following crucial information.
\begin{cor}
\label{monotone independence cor}
The position operators $(G_i)_{i\geq 1}$ are monotone independent in $(\cb, \om_{\Om})$.
\end{cor}

%%%%%%%%%%%%%%%%%%%%%%%%%%%%%%%%%%%%%%%%%%%%%%%%%%%%%%%%%%%%%%%%%%%%%%%%%%%%%%%%%%%%
\subsection{Partitions of a finite set}
\label{sec2b}

Let $S$ be a non empty linearly ordered finite set. The collection $\pi=\{B_1,\ldots, B_p\}$ is a \emph{partition} of the set $S$ if, for any $1\leq i \leq p$, one has
\begin{equation*}
B_i\cap B_j=\emptyset \quad \text{if $i\neq j$,} \qquad \bigcup_{i=1}^pB_i=S.\\
\end{equation*}
$B_i$ are called blocks of the partition $\pi$. The number of blocks of $\pi$ is denoted by $|\pi|$ and the set of all the partitions of $S$ is $P(S)$. In the special case $S=[m]:=\{1,\ldots,m\}$ we write the set of the partitions on it as $P(m)$. A partition $\s$ has \emph{crossing} if it contains at least two distinct blocks $B_i$ and $B_j$, and elements $v_1,v_2\in B_i$, $w_1,w_2\in B_j$ s.t. $v_1<w_1<v_2<w_2$. Otherwise, it has \emph{no crossing}. It is called a \emph{pair partition} if each block $B_h$ contains exactly two elements. In this case, for any $h$ we write $B_h=(l_h,r_h)$, where $l_h<r_h$, $l_1<l_2<\ldots< l_{|S|/2}$ and $|S|$ is the (necessarily even) number of elements in $S$.
Once $m$ is even, say $m=2n$, then $P_2(2n)$ and $NC_2(2n)$ will denote the sets of pair partitions and non crossing pair partitions (i.e partitions without any crossing), respectively. Each $\p\in P_2(2n)$ can be simply denoted by $(l_h,r_h)_{h=1}^n$. The cardinality of $NC_2(2n)$ is the $n$-th Catalan number $C_n$, i.e.
\begin{equation*}
C_n:=\frac{1}{n+1}\binom{2n}{n}.
\end{equation*}
We recall that in the lattice $\bz^2$ a Dyck path is a path which starts at $(0,0)$, makes steps either of the form $(n,k)\rightarrow (n+1,k+1)$ or of the form $(n,k)\rightarrow (n+1,k-1)$, ends on the $x$-axis and never goes strictly below the $x$-axis.
One has that for every positive integer $n$, the number of Dyck paths with $2n$ steps is equal to the $C_n$ \cite{nicaspe}.

For a partition $\pi \in P(m)$, with $k$ blocks $\pi=\{B_1, \ldots , B_k\}$, one has a natural partial  order $\preceq_{\pi}$ on the blocks given by $B_i\preceq_{\pi}B_j$ if $\min B_i\leq \min B_j \leq \max B_j \leq \max B_i$, where $\min B$ (resp. $\max B$) denotes the minimal (resp. maximal) element of the block $B$.

For a finite subset $I\subset\mathbb{N}$ and a partition $\pi \in P(m)$, with $k$ blocks $\pi=\{B_1, \ldots , B_k\}$ one defines the \emph{label function} $L: \pi \rightarrow I$ by $L(B_j)\in I$ for $1\leq j \leq k$, and call the pair $(\pi, L(\pi))$a \emph{labeled partition with labels} $L(\pi):=\{L(B_1), \ldots , L(B_k)\}$.
\begin{defin}
\label{weakmonlab}
The label function $L$ is \emph{weakly monotonic}
if it preserves the ordering of the blocks, i.e. if
\[
B_i\preceq_{\pi} B_j \Longrightarrow L(B_i)\leq L(B_j)
\]
\end{defin}
Using the above notations, the set of weakly monotonic ordered labeled partitions is denoted by $PWMO(I,m)$. As soon as $\pi\in NC_2(2n)$, we will use the notation $NC_2WMO(I,2n)$.

\bigskip

We end the subsection by recalling the following notations introduced in \cite{QCLMono}, which will be useful in the next. For $n,N\in \mathbb{N}$ with $1\leq n\leq N$, we take
\begin{equation*}
\mathfrak{M}_p(2n,N):=\{k:\{1,\ldots,2n\}\rightarrow \{1,\ldots N\} \mid |k^{-1}(j)|=2, j\in R(k)\}
\end{equation*}
as the set of all $2$-$1$ maps with range $R(k)$ included in $\{1,\ldots N\}$.

If moreover $(l_h,r_h)_{h=1}^n$ is a pair partition on $\{1,\ldots, 2n\}$, we denote by $\mathfrak{M}_p((l_h,r_h)_{h=1}^n,N)$ the collection of $k$ in $\mathfrak{M}_p(2n,N)$ s.t. $k(l_h)=k(r_h)$ for any $h$. Very often in the sequel, the generic $k(l)$ will be simply denoted as $k_l$ without further mentioning.

\subsection{Cauchy transform of a measure}
\label{sec2c}

In the last part of this section we briefly mention some basic facts about the Cauchy Transform of a probability measure. Let $\m$ be a probability measure defined on the Borel $\s$-field over $\br$. The moment sequence associated with $\mu$ is denoted by $(m_n(\mu))_{n\geq1}$. For each $z\in\bc$
$$
\cam(z):=\sum_{n=0}^{\infty}z^nm_n(\mu)
$$
is called moment generating function, which is considered as a formal power series if the series is not absolutely convergent.

From now on $\mathbb{C}^+$ and $\mathbb{C}^-$ will be the the upper and lower complex half-planes, respectively.  The Cauchy transform of $\m$ is defined as
\begin{equation}
\label{cauchydef}
\cg_{\mu}(z):=\int_{-\infty}^{+\infty}\frac{\mu(dx)}{z-x}.
\end{equation}
The map
$$
H_{\mu}(z):=\frac{1}{\cg_{\mu}(z)}
$$
is called the reciprocal Cauchy transform of $\m$.
$\cg_{\mu}(z)$ is analytic in $\bc\setminus \supp({\mu})$, and since $\cg_{\mu}(\overline{z})=\overline{\cg_{\mu}(z)}$, we can restrict its domain on $\mathbb{C}^+\cup \mathbb{R}$. In this region it uniquely determines $\m$, in the way below summarized (see e.g. \cite{horaob}).

\noindent (1) The limit
\begin{equation}
\label{stiedef}
-\frac{1}{\pi}\lim_{y\to 0^{+}}\Im\cg_{\mu}(x+iy)
\end{equation}
exists for a.e. $x\in\br$. If $\gpg(x)$ is defined as the above limit if this limit exists, and $0$ otherwise, then $\gpg(x)\,dx$ is the absolutely continuous part of $\mu$.

\noindent (2) $\cg_{\mu}(z)$ has a pole in $z=a\in\mathbb{R}$ if and only if $a$ is an isolated point of $\supp({\mu})$. In this case
$$
\m=c\d_a+(1-c)\nu,\,\,\,\,\, 0\leq c \leq 1
$$
and $\nu$ is a probability measure for which $\supp({\nu})\cap\{a\}=\emptyset$. Furthermore, $c=\text{Res}_{z=a}\cg_{\mu}(z)$.

The formula \eqref{stiedef} is called \emph{Stieltjes inversion formula}.

We finally report the moment generating function and the Cauchy transform for the standard (i.e supported on $[-2,2]$) Wigner semicircle law, respectively
\begin{align}
\begin{split}
\label{gfCtWigner}
\cam_1(z)&=\frac{1-\sqrt{1-4z^2}}{2z^2} \\
\cg_1(z)&=\frac{1}{2}\bigg(z-z\sqrt{1-\frac{4}{z^2}}\bigg),
\end{split}
\end{align}
the latter being recovered from the first by the well known identity
$$
\cg_{\mu}(z)=\frac{1}{z}\cam\bigg(\frac{1}{z}\bigg).
$$
We end the section by recalling the following result.
\begin{thm}
\cite{Mu}
\label{conv}
Let $a_{1},a_{2},\ldots,a_{n}\in \ga$ be monotonically independent self-adjoint
random variables, in the natural order, over a $(C^*)$ $*$-algebra $\ga$ with a state $\f$. If $\m_{a_i}$ is the probability distribution of ${a_i}$  under the state $\f$, then
$$
H_{\m_{a_1+a_2+\ldots + a_n}}(z)=H_{\m_{a_1}}(H_{\m_{a_2}}(\cdots H_{\m_{a_n}}(z)\cdots)).
$$
\end{thm}
Moreover, Theorem 3.5 in \cite{Mu} ensures that for any pair of probability measures $\m,\nu$ on $\mathbb{R}$, there exists a unique
distribution $\r$ on $\mathbb{R}$ such that
$$
H_{\r}(z) = H_\m(H_\nu(z)).
$$
$\r$ is called the monotonic convolution of $\m$ and $\nu$.

\section{Moments for field position operators}
\label{sec3}

In this section we will investigate the vacuum distribution for sums $\sum_{k\in I}G_k$ of position operators $G_k=A_k+A^{\dag}_k$, where $I\subset \bn$, $|I|<\infty$, and $|\cdot|$ denotes the cardinality.

To this end we define
$$
\mu_{I,n}:=\om_{\Om}\bigg(\bigg(\sum_{k\in I} G_k\bigg)^{n}\bigg)
$$
the $n$-th moments for such sums. Whenever $I=[m]$ we use the notation $\mu_{m,n}$.
For each $i\geq 1$, as in the previous section it is useful to denote creators and annhilators as $A_i^{1}:=A_i^{\dag}$ and $A_i^{-1}:=A_i$, respectively.

We first treat the case $m=1$. If $\eps=(\eps(1),\ldots,\eps(n))$, $\eps(j)\in\{-1,1\}$ for any $j=1,\ldots,n$, then
\begin{equation}
\label{momG1}
\mu_{1,n}:=\om_{\Om}((G_1)^n)=\sum_{(\eps(1),\ldots,\eps(n))\in\{-1,1\}^n}\om_{\Om}(A_1^{\eps(1)}\cdots A_1^{\eps(n)}).
\end{equation}
If $n$ is odd the vacuum expectation above is null. Therefore from now on we will consider only even moments. If, instead $n$ is even, say $2n$, by an abuse of notation, both the following conditions are necessary for the non vanishing of \eqref{momG1}

\bigskip

\noindent (1) $\sum_{j=1}^{2n}\eps(j)=0$,

\bigskip
\noindent (2) $\sum_{j=h}^{2n}\eps(j)\geq 0$, for $h=1,\ldots, 2n$.

\bigskip

\noindent From now on, we will use $\{-1,1\}^{2n}_+$ to denote the entire collections of strings $\eps\in \{-1,1\}^{2n}$ satisfying (1) and (2) above.
\begin{prop}
\label{wig}
The distribution measure of $G_1$ with respect to the vacuum state $\om_{\Om}$ is the standard Wigner law.
\end{prop}
\begin{proof}
From \eqref{momG1}, $(1)$ and $(2)$ above, it is sufficient to prove
$$
\sum_{\eps\in\{-1,1\}^{2n}_+}\om_{\Om}(A_1^{\eps(1)}\cdots A_1^{\eps(2n)})=C_n
$$
for any $n$. Indeed, let $\{l_1,\ldots,l_p\}$ be the (possibly empty) subset of $[2n]$ such that $\eps(l_j)=-1$, for any $j$, and $l_1<\cdots<l_p$.
For the special case $\varepsilon\in\{-1,1\}^{2n}_+$, conditions (1) and (2) above immediately entail $p=n$, $l_1=1$
and $l_n<2n$. In addition, to each $\varepsilon\in\{-1,1\}^{2n}_+$ one can
uniquely associate a non crossing pair partition of the set consisting of $2n$ elements. Namely, the first block is obtained just pairing the first
consecutive $(-1,1)$ appearing on the string starting from the left. The
second pairing will arise by canceling the two indices previously paired,
and then reproducing the previous scheme to the remaining ones, and so on. As $|NC_{2}(2n)|=C_n$, the thesis follows after noticing that, for any $\eps\in\{-1,1\}^{2n}_+$, $\om_\Om(A_1^{\eps(1)}\cdots A_1^{\eps(2n)})=1$ as a consequence of Lemma \ref{3a}.
\end{proof}
Since the one-to-one correspondence described in the proof above between $\eps\in\{-1,1\}^{2n}_+$ and $\pi:=(l_h,r_h)_{h=1}^n\in NC_{2}(2n)$, from now on we will often use the natural identification $\eps\equiv (l_h,r_h)_{h=1}^n$.

The arguments used in the proof of Proposition \ref{wig} automatically give the field position operators $G_i$ are identically distributed w.r.t. the vacuum for each $i\geq 1$.

As a consequence of Corollary \ref{monotone independence cor} and Proposition \ref{wig}, one obtains that investigating the distribution of $\sum_{i=1}^m G_i$ coincides with finding the $m$-fold monotonic convolution of the standard Wigner law with itself. We point out that, reasoning as in the case $m=1$, one finds that for any $m$, $\m_{[m],2n+1}=0$.

\subsection{Distribution of sum of two position operators}
\label{sec3.1}
In this subsection we consider the case $m=2$, i.e. $G_1+G_2$.

The next result gives us that the vacuum moments for $G_1+G_2$ can be computed by counting the weakly monotone ordered non crossing pair partitions of the set $[2n]$. This automatically entails all the $\sum_{k\in I} G_k$, $|I|=2$ are identically distributed under the state $\om_\Om$. Although it is a particular case of the successive Theorem \ref{momcasem}, we put here its direct proof for the convenience of the reader.
\begin{prop}
\label{momcase2}
For any $n\geq1$, one has
\begin{equation}
\label{moments2a}
\mu_{2,2n}=\om_{\Om}\bigg(\big(G_1+G_2\big)^{2n}\bigg)=\big|NC_2WMO([2],2n)\big|\,.
\end{equation}
In addition, if $d_n:=\mu_{2,2n}$, the following recursive formula holds
\begin{equation}
\label{dn}
d_n=\sum_{k=1}^{n}d_{n-k}(d_{k-1}+C_{k-1}),\,\,\,\,\,\,\,\,  d_0=1,
\end{equation}
where the $C_k$ are the Catalan numbers.
\end{prop}
\begin{proof}
We first prove \eqref{moments2a}. Arguing as in Proposition \ref{wig}, one immediately obtains
\begin{equation}
\label{thmfor1}
\om_{\Om}\bigg(\big(G_1+G_2\big)^{2n}\bigg)=\sum_{(l_h,r_h)_{h=1}^n\in NC_2(2n)}
\sum_{k_1,\ldots,k_{2n}\in[2]}\om_{\Om}(A_{k_1}^{\eps(1)}\cdots A_{k_{2n}}^{\eps(2n)}),
\end{equation}
where we used the natural identification of $\eps\in\{-1,1\}^{2n}_+$ with $(l_h,r_h)_{h=1}^n$. In this case one replaces $\{k_1,\ldots, k_{2n}\}$ with $\{k_{l_1},k_{r_1},\ldots, k_{l_n},k_{r_n}\}$. Reasoning as in \cite{QCLMono}, Lemma 3.3., one finds the r.h.s. of \eqref{thmfor1} is not automatically null only when $k_{l_h} = k_{r_h}$ for any $h = 1,\ldots,n$.
Further, from \eqref{jgeqk}, \eqref{alpha} and definition of creators and annihilators, for each $j,k\geq1$, it follows
\begin{equation}
\label{AAdagAAdag}
A_jA_j^{\dag}A_kA_k^{\dag}=A_lA^{\dag}_l \qquad l:=\min\{j,k\}
\end{equation}
and
\begin{equation}
\label{AAAdagAdag}
 A_jA_kA_k^{\dag}A_j^{\dag}=\a_{k,j}A_jA^{\dag}_j,
 \end{equation}
where $\a_{k,j}$ is defined in Lemma \ref{3a}. This gives that, for fixed $\eps\in\{-1,1\}^{2n}_+$ and $k_1,\ldots, k_{2n}\in [2]$, any $\om_{\Om}(A_{k_1}^{\eps(1)}\cdots A_{k_{2n}}^{\eps(2n)})$, if not null, is reduced to
$$
\om_{\Om}(A_{k_h} A_{k_{h}}^{\dag})=1,
$$
where $k_h:=\min\{k_{l_j} \mid k_{l_j}=k_{r_j}, j=1,\ldots,n\}$.
Thus it is sufficient to prove that only the non crossing pair partitions which are labeled in the weakly monotone order survive into the sums in \eqref{thmfor1}.

Indeed, if we take $\eps=(l_h,r_h)_{h=1}^n$ as the interval partition (i.e. $r_h=l_{h+1}$ for any $h$), we have no inner nor outer blocks, i.e. $\eps$ is automatically weakly monotone ordered. As a consequence we suppose $\eps=(l_h,r_h)_{h=1}^n$ is not an interval partition, but still non crossing. Let $j$ be the minimum in $[n]$ for which $r_j\neq l_j+1$ and consider the (non crossing) pair partition $\pi_j$ given by all the blocks $B_h:=(l_h,r_h)$ for which $l_j<l_h<r_h<r_j$. If $\pi_j$ is an interval partition from \eqref{AAdagAAdag} and \eqref{AAAdagAdag}, it follows
\begin{equation*}
A_{k_{l_j}}A_{k_{l_{j+1}}}A_{k_{l_{j+1}}}^\dag \cdots A_{k_{l_j}}^{\dag}=\a_{k_{l_h}, k_{l_j}}A_{k_{l_j}}A_{k_{l_j}}^{\dag},
\end{equation*}
where $k_{l_h}=\min\{k_{l_p} \mid B_p\in \pi_j\}$. Therefore the partition $\pi:= ((l_j,r_j),\pi_j)$ has to be weakly monotone ordered.
If instead $\pi_j$ is not an interval partition, we take any block $B_p:=(l_p,r_p)$ in $\pi_j$ s.t. $r_p\neq l_p+1$ and argue as above. Iterating the same procedure for each block with no consecutive indices in the partition induced by $\eps$, \eqref{moments2a} follows.

Finally we show that the number of non crossing weakly monotone ordered pair partitions with 2 possible labels satisfy \eqref{dn}. In fact, it trivially holds when $n=1$.

\noindent Assume now \eqref{dn} is true for any $s<n$ and fix $\eps\in\{-1,1\}_{+}^{2n}$. Here
\begin{equation*}
h:=\min\big\{m>0 \mid \sum_{i=1}^{m}\eps(i)=0\big\}
\end{equation*}
is a well defined and even integer. As a shorthand notation we put $\cd_n:=NC_2WMO([2],2n)$. Moreover, we denote by $\cd_n^{k}$ the non crossing pair partitions in $NC_2WMO([2],2n)$ such that the sum of $\eps(i)$ vanishes for the first time after $2k$ terms. Since any element in $\cd_n^{k}$ belongs to $\{-1,1\}_{+}^{2n}$, one finds
\begin{equation}
\label{Dnkset}
\cd_n^{k}=\big\{ \eps\in NC_2WMO([2],2n)\mid
 \sum_{i=1}^{2k}\eps(i)=0 \, \text{and} \,  \sum_{i=1}^{l}\eps(i)<0\, \forall l<2k\big\}.
\end{equation}
As a consequence
\begin{equation*}
\cd_n=\bigcup_{k=1}^{n}\cd_n^{k}
\end{equation*}
and, therefore, since the $\cd_n^{k}$ are pairwise disjoint,
\begin{equation*}
\big|\cd_n\big|=\sum_{k=1}^n \big|\cd_n^k\big|.
\end{equation*}
By \eqref{moments2a} one has $d_n=|\cd_n|$ and, further we put $d_n^k:=|\cd_n^k|$. Then \eqref{dn} follows as soon as we prove the following equality
\begin{equation}
\label{dnk}
d_n^k=d_{n-k}(d_{k-1}+C_{k-1})\,.
\end{equation}
To this aim, fix $k=1,\ldots,n$, and split the set $S=\{1,\ldots,2n\}$ into two subsets $S=S^{\prime}\cup S^{\prime\prime}$, where $S^{\prime}=\{1,\ldots,2k\}$ and $S^{\prime\prime}=\{2k+1,\ldots,2n\}$. Then
\begin{equation*}
 d_n^k =\big|NC_2WMO^{\prime}([2],S^{\prime})\big|\cdot \big|NC_2WMO([2],S^{\prime\prime})\big|,
\end{equation*}
where $NC_2WMO^{\prime}([2],S^{\prime})$ is the subset of $NC_2WMO([2],S^{\prime})$ given by the partitions $\pi=\{B_1,\ldots,B_k\}$ for which $B_1=(1,2k)$. Here, $|NC_2WMO([2],S^{\prime\prime})|=| NC_2WMO([2],2(n-k))|=d_{n-k}$ by induction assumption. Notice that for each $h\in\{2,\ldots,k\}$ any $B_h\in\pi$ is inside $B_1$.
We write $L(B_1)=i_1$ and $L(B_h)=i_h$ for each $h$. The following two cases are allowed for the label $i_1$

\bigskip

\noindent (1) $i_1$=1. Here all the labels $i_h$ are allowed for the blocks $B_h$, and thus the cardinality of $NC_2WMO^{\prime}([2],S^{\prime})$ reduces to $|NC_2WMO([2],S^{\prime})|=d_{k-1}$ by induction assumption.

\bigskip

\noindent (2) $i_1=2$. In this circumstance $i_h=2$ for each $h=2,\ldots,k$ since $\pi$ is weakly monotone ordered.\\
Therefore in such a case $|NC_2WMO^{\prime}([2],S^{\prime})|$ coincides with $|NC_2(2k-2)|=C_{k-1}$.

Summing up, one has
$$
\big|NC_2WMO^{\prime}([2],S^{\prime})\big|=d_{k-1}+C_{k-1}
$$
and \eqref{dnk} follows.
\end{proof}
The next result gives us the Cauchy transform and the moment generating function of the vacuum distribution of $G_1+G_2$, which are the building blocks to achieve the distribution itself. The proof can be performed by means of the monotone convolution or, as a consequence of the previous result, directly using the recurrence formula for moments.
\begin{lem}
\label{momcau}
The moment generating function and the Cauchy transform for $G_1+G_2$ are respectively given by
\begin{equation}
\label{genfunG1+G2}
\cam_2(z)=\frac{1-z^2\cam_1(z)-\sqrt{(z^2\cam_1(z)-1)^2-4z^2}}{2z^2}
\end{equation}
and
\begin{equation}
\label{cauchyG1+G2}
\cg_2(z)= \frac{1}{2}\bigg(z-\cg_1(z)\mp \sqrt{(\cg_1(z)-z)^2-4}\bigg),
\end{equation}
where $\cam_1(z)$ and $\cg_1(z)$ are defined in \eqref{gfCtWigner}.
\end{lem}
\begin{proof}
As the vacuum law for $G_1+G_2$ is symmetric, we only deal with even moments. From \eqref{dn}, we have
\begin{equation*}
\begin{split}
  \cam_2(z)=&1+\sum_{n=1}^{\infty}z^{2n}\bigg(\sum_{k=1}^n d_{n-k}(d_{k-1}+C_{k-1})\bigg) \\
         =& 1+\sum_{n=1}^{\infty}z^{2n}\bigg(\sum_{k=1}^n d_{n-k}d_{k-1}\bigg)+\sum_{n=1}^{\infty}z^{2n}\bigg(\sum_{k=1}^n d_{n-k}C_{k-1}\bigg)\\
         =&1+z^2\sum_{n=1}^{\infty}\sum_{k=1}^nz^{2(n-k)}d_{n-k}z^{2(k-1)}d_{k-1} \\
         &+z^2\sum_{n=1}^{\infty}\sum_{k=1}^nz^{2(n-k)}d_{n-k}z^{2(k-1)}C_{k-1}\\
         =&1+z^2 \sum_{k=1}^{\infty}z^{2(k-1)}d_{k-1}\sum_{n=k}^{\infty}z^{2(n-k)}d_{n-k}\\
         &+z^2\sum_{k=1}^{\infty}z^{2(k-1)}C_{k-1}\sum_{n=k}^{\infty}z^{2(n-k)}d_{n-k}\\
         =&1+z^2\cam_2(z)^{2}+z^2\cam_2(z)\cam_1(z).
\end{split}
\end{equation*}
As a consequence,
\begin{equation*}
    z^2\cam_2(z)^2+\cam_2(z)(z^2\cam_1(z)-1)+1=0
\end{equation*}
and \eqref{genfunG1+G2} follows after taking into account that $\lim_{z\to0}\cam(z)=d_0=1$.
Exploiting the usual identity $\cg_2(z)=\frac{1}{z}\cam_2\big(\frac{1}{z}\big)$, \eqref{genfunG1+G2} gives
\begin{equation*}
\begin{split}
\cg_2(z)&= \frac{1}{2}\bigg(z-\cg_1(z)-z\sqrt{\frac{(\cg_1(z)-z)^2-4}{z^2}}\bigg)\\
   &=\frac{1}{2}\bigg(z-\cg_1(z)\mp \sqrt{(\cg_1(z)-z)^2-4}\bigg)\,.
   \end{split}
\end{equation*}
\end{proof}
We observe that \eqref{cauchyG1+G2} can be also obtained by the relation $\cg_2(z)=\cu\big(\frac{1}{\cu(z)}\big)$, i.e. by using Theorem \ref{conv}.

The Stieltjes inversion formula \eqref{stiedef} gives the absolutely continuous part of the law of $G_1+G_2$. Its computation needs some preliminary arguments. We start with the evaluation of the limits for the map $z\mapsto z\sqrt{1-\frac{4}{z^2}}$. It seems convenient to apply some arguments which appeared in \cite{Wysdef}, here reported for the reader's convenience. Let $x$ and $y$ the real and imaginary parts of $z$, respectively, and $a+ib:=z\sqrt{1-\frac{4}{z^2}}$. Then, since $\cg_1$ maps $\mathbb{C}^{+}$ into $\mathbb{C}^-$, from \eqref{gfCtWigner} it follows that $y > 0$ implies $b > 0$. Furthermore, since $ab = xy$, $a$ and $x$ have the same sign. This entails the functions $a = a(x, y)$ and $b = b(x, y)$ have the following limits.

\bigskip

\noindent (1) If $|x|\leq2$, then $\lim_{y\to0^{+}} a(x; y) = 0$ and $\lim_{y\to0^{+}}b(x; y) =\sqrt{4-x^2}$;

\bigskip

\noindent (2) If $x>2$, then $\lim_{y\to0^{+}} a(x; y) =\sqrt{x^2-4}$ and $\lim_{y\to 0^{+}}b(x; y) = 0$;

\bigskip

\noindent (3) If $x<-2$ then $\lim_{y\to0^{+}} a(x; y) =-\sqrt{x^2-4}$  and $\lim_{y\to 0^{+}}b(x; y) = 0$.

\bigskip

We are now able to prove the main result for the sum of two position operators. As usual we reduce the matter to the case $G_1+G_2$.
\begin{thm}
\label{dens}
The distribution for $G_1+G_2$ is absolutely continuous with the density $\gpg$ given by
\begin{equation}
\label{density}
\gpg(x)=\begin{cases}
\frac{1}{4\pi}\big(\sqrt{\sqrt{100-16x^2}-x^2+10}-\sqrt{4-x^2}\big) &\text{if $|x|\leq2$}\\
\frac{1}{4\pi}\sqrt{-2x^2-2|x|\sqrt{x^2-4}+20} &\text{if $\displaystyle 2\leq |x|\leq\frac{5}{2}$}\\
0 & \text{if $\displaystyle |x|\geq\frac{5}{2}$}.
\end{cases}
\end{equation}
\end{thm}
\begin{proof}
Here we prove that the density of $G_1+G_2$ is given in \eqref{density}, and postpone the proof of the absence of atoms in Theorem \ref{lemabscon}.\\
We first compute $\Im{\cg_2(z)}$. As above we use the following change of variables $x+iy:=z$ and $a+ib:=z\sqrt{1-\frac{4}{z^2}}$. Thus $2\cg_1(x+iy) = (x-a)+i(y-b)$, and in \eqref{cauchyG1+G2} we have
\begin{equation*}
\cg_2(x+iy)= \frac{x+a+i(y+b)}{4}\mp\frac{1}{4}\sqrt{[x+a+i(y+b)]^2-16}\,.
\end{equation*}
For the imaginary part of the square root above, after defining
\begin{align*}
&A:=(x+a)^2-(y+b)^2-16 \\
&B:=2(x+a)(y+b),
\end{align*}
it follows
\begin{equation*}
\label{p+iq}
P+iQ:=\sqrt{[x+a+i(y+b)]^2-16}=\sqrt{A+iB}.
\end{equation*}
As a consequence, $P$ and $Q$ satisfy the following identities
\begin{align*}
&P=\frac{1}{\sqrt{2}}\sqrt{\sqrt{A^2+B^2}+A}, \\
&Q=\frac{\sgn B}{\sqrt{2}}\sqrt{\sqrt{A^2+B^2}-A}.
\end{align*}
As previously noticed $y+b>0$ and $a$ and $x$ have the same sign. Thus $\sgn B=\sgn{x}$ and one has
\begin{align}
\begin{split}
\label{imcauchy}
\Im{\cg_2(x+iy)}&=\frac{y+b}{4}\mp \frac{Q}{4} \\
&=\frac{y+b}{4}\mp \frac{\sgn{x}}{4\sqrt{2}}\sqrt{\sqrt{A^2+B^2}-A}.
\end{split}
\end{align}
Since both $y$ and $b$ are positive, the imaginary term in $\cg_2(x+iy)$ belongs to the complex lower half plane if and only if we choose "$-$" when $x>0$ and "$+$" when $x<0$, on the r.h.s. of \eqref{imcauchy}. Consequently
\begin{equation*}
\Im{\cg_2}(x+iy)=\frac{y+b}{4}-\frac{1}{4\sqrt{2}}\sqrt{\sqrt{A^2+B^2}-A}.
\end{equation*}
If $C:=\sqrt{A^2+B^2}$, one has
$$
\Im{\cg_2}(x+iy)=\frac{y+b}{4}-\frac{1}{4\sqrt{2}}\sqrt{C- (x-a)^2+(y+b)^2+16}.
$$
Taking into account the limits for the functions $a(x,y)$ and $b(x,y)$ as computed above, the following cases may occur. Here we consider only the case $x\geq 0$, as the density is an even function.

\bigskip

\noindent (1) $0<x\leq2$. Then from \eqref{stiedef}
\begin{align*}
&\lim_{y\to 0^+}\Im{\cg_2(x+iy)} \\
&=\frac{\sqrt{4-x^2}}{4}-\frac{1}{4}\sqrt{\sqrt{100-16x^2}-x^2+10}.
\end{align*}
For $x=0$ we have
$$
\cg_1(iy)=\frac{i}{2}\big(y-\sqrt{y^2+4}\big)
$$
and, by choosing "$-$" in \eqref{cauchyG1+G2}
$$
\cg_2(iy)=\frac{i}{4}\bigg(y+\sqrt{y^2+4}-\sqrt{2y^2+2y\sqrt{y^2+4}+20}\bigg).
$$

Then $\lim_{y\to0}\Im{\cg_2(iy)}=\frac{1}{2}-\frac{1}{4}\sqrt{20}=\gpg(0)$.

\bigskip

\noindent (2) $x>2$. In this case
\begin{equation*}
 \lim_{y\to 0^+}\Im{\cg_2(x+iy)} =-\frac{1}{4}\sqrt{\big|10-x^2-x\sqrt{x^2-4}\big|-x^2-x\sqrt{x^2-4}+10}.
\end{equation*}
Therefore one has two subcases

\bigskip

\noindent (2a) $2<x<\frac{5}{2}$, which gives
\[
 \lim_{y\to 0^+}\Im{\cg_2(x+iy)} = -\frac{1}{4}\sqrt{20-2x^2-2x\sqrt{x^2-4}}\,.
 \]
\newline
\noindent (2b) $x\geq\frac{5}{2}$, which entails
$$
\lim_{y\to 0^+}\Im{\cg_2(x+iy)}=0\,.
$$
The thesis then follows from \eqref{stiedef}.
\end{proof}
In Figure \ref{fig:dens} one finds the plot of $\gpg$.
\begin{figure}[b!]
  \includegraphics[scale=0.2]{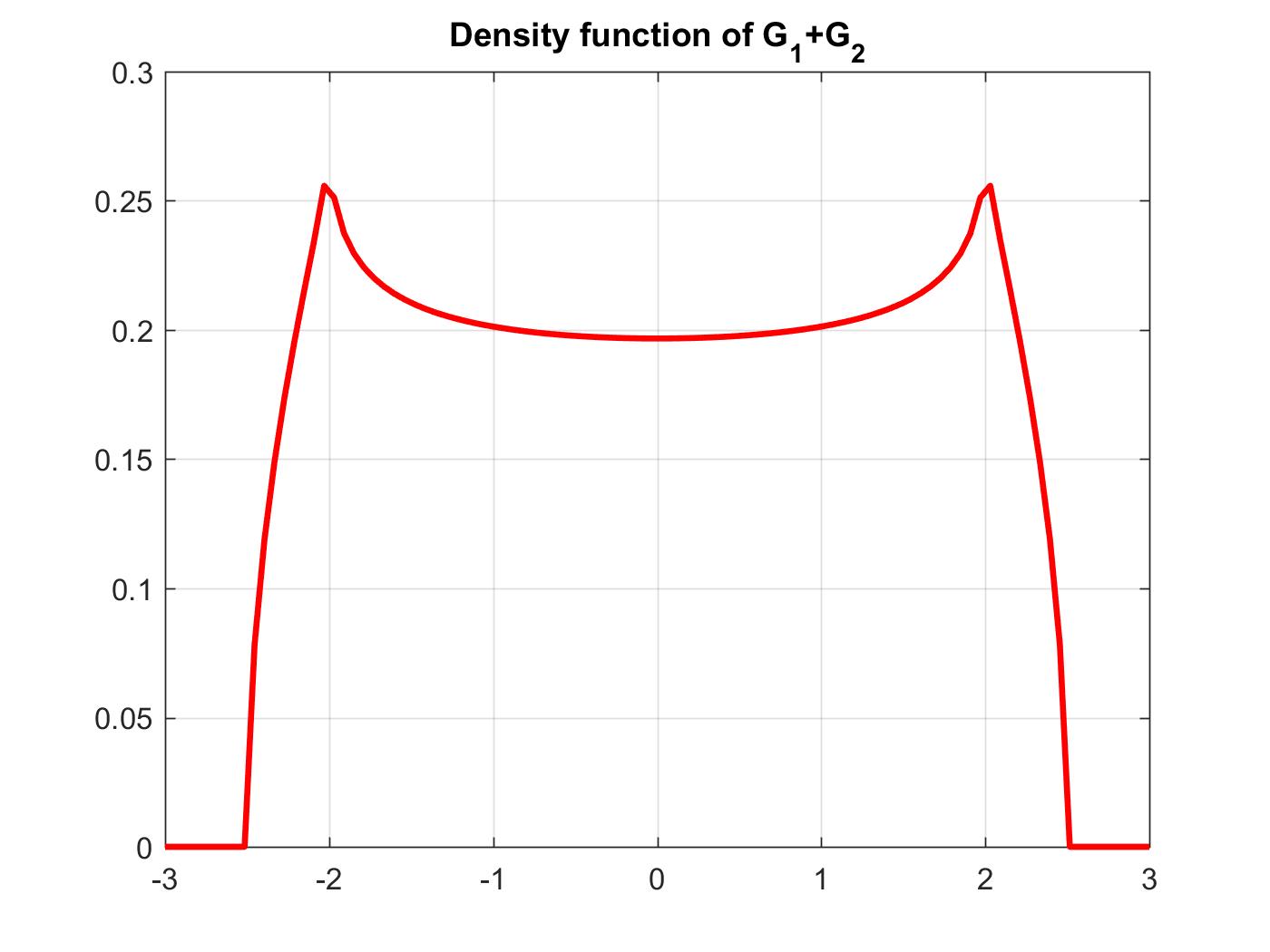}
  \caption{}
  \label{fig:dens}
\end{figure}

\subsection{On the distribution of sum of $m\geq 3$ position operators}
\label{sec3.2}
Now we turn to the case $m>2$, i.e. we deal with the problem of finding the vacuum distribution for $G_1 +\cdots + G_m$. We start with the generalization of Proposition \ref{momcase2}.
\begin{thm}
\label{momcasem}
For any $m,n\geq1$, one has
\begin{equation}
\label{moments2}
\mu_{m,2n}:=\om_{\Om}\bigg(\bigg(\sum_{k=1}^mG_k\bigg)^{2n}\bigg)=|NC_2WMO([m],2n)|.
\end{equation}
In addition, if one denotes $d_n^{(m)}:=\mu_{m,2n}$, then
\begin{equation}
\label{dn^m}
d_n^{(m)}=\sum_{k=1}^{n}d_{n-k}^{(m)}\sum_{j=1}^md_{k-1}^{(j)},
\end{equation}
where $d_0^{(j)}=1$ for each $j$ and $d_n^{(1)}=C_n$ for any $n$.
\end{thm}
\begin{proof}
One notices that, \emph{mutatis mutandis}, \eqref{moments2} is achieved reasoning as in Proposition \ref{momcase2}.
We then prove \eqref{dn^m} by an induction procedure on $n$.
Indeed, for $n=1$, \eqref{moments2} and \eqref{cr} give
$$
d_{1}^{(m)}= \big|NC_2WMO([m],2)\big|=m=
d_0^{(m)}\sum_{j=1}^md_{0}^{(j)}\,.
$$
Suppose \eqref{dn^m} holds for each $r<n$. We now check the identity for $r=n$. To this aim, for any $k=1,\ldots, n$ one defines
\begin{equation}
\label{Dnkm}
\cd_n^{k,(m)}:=\big\{\eps\in NC_2WMO([m],2n) \mid
 \sum_{i=1}^{2k}\eps(i)=0,  \sum_{i=1}^{l}\eps(i)<0\,\, \text{for}\,\, l<2k\big\}
\end{equation}
and denotes by $d_n^{k,(m)}$ its cardinality. As the sets $\cd_n^{k,(m)}$ are pairwise disjoint and their union is the whole collection of elements in $NC_2WMO([m],2n)$, it follows
$$
d_n^{(m)}=\sum_{k=1}^{n}d_n^{k,(m)}.
$$
Now we show that for any $k=1,\ldots,n$
\begin{equation*}
d_n^{k,(m)}=d_{n-k}^{(m)} \sum_{j=1}^m d_{k-1}^{(j)}.
\end{equation*}
Indeed, we split $S=\{1,\ldots,2n\}$ into the two subsets $S=S^{\prime}\cup S^{\prime\prime}$, where $S^{\prime}=\{1,\ldots,2k\}$ and $S^{\prime\prime}=\{2k+1,\ldots,2n\}$. Therefore
\begin{equation*}
d_n^{k,(m)}= \big|NC_2WMO^{\prime}([m],S^{\prime}) \big|\cdot  \big|NC_2WMO([m],S^{\prime\prime}) \big|,
\end{equation*}
where, as in the previous case, $NC_2WMO^{\prime}([m],S^{\prime})$ is the subset of $NC_2WMO([m],S^{\prime})$ given by the partitions $\pi=\{B_1,\ldots,B_k\}$ for which $B_1=(1,2k)$.\\
Now $|NC_2WMO([m],S^{\prime\prime})|=| NC_2WMO([m],2n-2k)|=d_{n-k}^{(m)}$ by induction assumption.

Notice that for each $h\in\{2,\ldots,k\}$ any $B_h\in\pi$ is inside $B_1$.
As in Proposition \ref{momcase2}, we label $L(B_1)= i_1$ and $L(B_h)= i_h$ for each $h$.

\noindent By definition $i_1=l$, for each $l=1,\ldots, m$. The weakly monotone ordering gives $i_h=l+r$, $r=0,\ldots, m-l$, and the number of partitions allowed $\big| NC_2WMO^{\prime}([m-l+1],2k-2)\big|$ reduces to
$$
\big| NC_2WMO([m-l+1],2k-2)\big|=d_{k-1}^{(m-l+1)}\,,
$$
the last equality coming from the induction assumption. As a consequence,
$$
\big|NC_2WMO^{\prime}([m],S^{\prime})\big|=\sum_{j=1}^m d_{k-1}^{(j)}.
$$
\end{proof}
In the next proposition we will recall and prove some properties of $d_n^{(m)}$ for any $n\geq0$.
\begin{prop}
\label{momentspoly}
For any $n\geq0$, $d_n^{(m)}$ is a polynomial of degree $n$ in the variable $m$ which does not contain constants when $n\geq1$. Moreover, $m=0$ is a root of multiplicity $1$.
\end{prop}
\begin{proof}
Since \cite{HasSa}, Proposition 4.4, we just need to show the degree of the polynomial is $n$ and the last sentence in the statement.
The first assertion is true for $n=0$ and $n=1$ as $d_0^{(m)}=1$ and $d_1^{(m)}=m$ for any $m$. Suppose by induction that the degrees of $d_{n-k}^{(m)}$ and $d_{k-1}^{(j)}$ are $n-k$ and $k-1$, with indeterminate $m$ and $j$, respectively. Therefore for any $k=1,\ldots, n-1$
\begin{equation*}
\sum_{j=1}^{m}d_{k-1}^{(j)}=\sum_{j=1}^{m}\sum_{h=1}^{k-1}\a_h^{(j)} j^{k-h},
\end{equation*}
where $\a_h^{(j)}$ are complex numbers.\\
Using the Faulhaber's formula with Bernoulli coefficients $B_p$ such that $B_1=\frac{1}{2}$, one finds
\begin{equation}
\label{fau}
\sum_{j=1}^md_{k-1}^{(j)}=\a_1^{(m)}m^k+\g(m),
\end{equation}
$\g(m)$ being a polynomial of degree less than $k$. The statement then follows from \eqref{dn^m}.

Notice that $d_1^{(m)}$ vanishes in $0$ and such a root has multiplicity 1. Again we use an induction procedure and suppose that for any $1\leq k\leq n-1$ and any $m$
$$
d_{n-k}^{(m)}=m q_{n-k-1}(m),\,\,\,\,\,\,\,\, q_{n-k-1}(0)\neq 0.
$$
We now show the following factorisation
$$
d_n^{(m)}=m q_{n-1}(m),
$$
where $q_{n-1}(m)$ is a polynomial of degree $n-1$ on $m$ for which $q_{n-1}(0)\neq 0$.

In fact, from \eqref{dn^m} and \eqref{fau},
\begin{align*}
d_n^{(m)}=&\sum_{k=1}^{n-1}d_{n-k}^{(m)}\sum_{j=1}^m d_{k-1}^{(j)}+d_0^{(m)}\sum_{j=1}^m d_{n-1}^{(j)} \\
=& \sum_{k=1}^{n-1}m q_{n-k-1}(m)[\a_1^{(m)}m^k+\g_1(m)]+\b_1^{(m)}m^n+\g_2(m) \\
=& m \bigg[\sum_{k=1}^{n-1}m q_{n-k-1}(m)[\a_1^{(m)}m^{k-1}+\theta_1(m)]+\b_1^{(m)}m^{n-1}+\theta_2(m)\bigg],
\end{align*}
where $\theta_1(m)$ and $\theta_2(m)$ are both polynomials with constant term of degrees less than $k-1$ and $n-1$, respectively.
The proof thus ends by taking
$$
q_{n-1}(m):=\bigg[\sum_{k=1}^{n-1}m q_{n-k-1}(m)[\a_1^{(m)}m^{k-1}+\theta_1(m)]+\b_1^{(m)}m^{n-1}+\theta_2(m)\bigg].
$$
\end{proof}
Now we shall establish some properties of the moments generating function and the Cauchy transform for the sum $G_1+\cdots+G_m$, denoted, in analogy with the case $m=2$, by $\cam_m$ and $\cg_m$ respectively. Namely, for each $z\in \bc$
$$
\cam_m(z)=\sum_{n=0}^{\infty}z^{2n}d_n^{(m)}
$$
and
\begin{equation}
\label{caucaum}
\cg_m(z)=\frac{1}{z}\cam_m\bigg(\frac{1}{z}\bigg).
\end{equation}
Notice that $\cam_1(z)=\cam_1(z)$ and $\cg_1(z)=\cg_1(z)$.
Moreover, for $z\in\mathbb{C}$ one denotes
$$
L_m(z):=\sum_{k=1}^m \cam_k(z)
$$
and
$$
K_m(z):=\sum_{k=1}^m \cg_k(z).
$$
\begin{thm}
\label{momcaucasem}
For any $m\geq1$, $z\in\bc$ one has
\begin{align}
z^2{\cam_m}^2(z)+\cam_m(z)(z^2L_{m-1}(z)-1)+1 &=0,\label{genfunm}\\
\cg_m^2(z)+\cg_m(z)(K_{m-1}(z)-z)+1&=0.\label{cauchym}
\end{align}
\end{thm}
\begin{proof}
Indeed, for $m\geq1$ and $z\in\bc$, we preliminary show that
\begin{equation}
\label{cauchym2}
z-K_{m-1}(z)=\cg_m(z)+\frac{1}{\cg_m(z)}.
\end{equation}
To this purpose, one firstly denotes
$$
\cf_m(z):=\frac{1}{\cg_m(z)}
$$
and recall that the random variables $G_1,\ldots, G_m$ are monotonically independent, from Theorem \ref{conv} one has $\cf_m(z)=\cf_{m-1}(\cf_1(z))$ and $\cg_m(z)=\cg_{m-1}\big(\frac{1}{\cg_{1}(z)}\big)$. Furthermore, \eqref{gfCtWigner} gives $\cf_1(z)=\frac{1}{2}(z+\sqrt{z^2-4})$,
one consequently one has
\begin{equation}
\label{recur}
\cg_1(z)+\frac{1}{\cg_1(z)}=z.
\end{equation}
Suppose now \eqref{cauchym2} holds for each $p<m$. In this case
\begin{align*}
\cg_m(z)+\frac{1}{\cg_m(z)}=& (\cg_{m-1}+\cf_{m-1})(\cf_{1}(z)) \\
=&\cf_{1}(z)- K_{m-2}(\cf_{1}(z)) \\
=&z-\cg_1(z)-\sum_{k=2}^{m-1}\cg_k(z),
\end{align*}
the last equality following from \eqref{recur}. Thus \eqref{cauchym2} holds, and it immediately gives \eqref{cauchym}. Therefore \eqref{genfunm} is straightforwardly obtained from \eqref{caucaum}.
\end{proof}
One notices \eqref{genfunm} and consequently \eqref{cauchym} can be directly derived from \eqref{dn^m}, without using monotone convolution.

From now on we denote by $\nu_m$  the distribution of $G_1+\cdots G_m$ w.r.t. the vacuum $\om_\Om$. This means that $\nu_1$ is the standard Wigner law and $\nu_2$ is the absolutely continuous measure computed in Theorem \ref{dens}. Formula \eqref{cauchym2} is fundamental for proving the following
\begin{thm}
\label{lemabscon}
For any $m\geq 1$ and $I\subset \mathbb{N}$ such that $|I|=m$, the vacuum distribution of $\sum_{k\in I} G_k$, and consequently the $m$-fold monotone convolution of the standard Wigner law is a compactly supported absolutely continuous measure on the real line.
\end{thm}
\begin{proof}
As usual it is enough to prove the statement when $I=[m]$. Let $\nu_m$ be the vacuum law of $G_1+\cdots + G_m$. From Lemma 6.4 in \cite{Mu}, $\nu_m$ has compact support. As a consequence of the Lebesgue decomposition theorem, we only need to prove that, for any $m$, the singular part of $\nu_m$ w.r.t. its absolutely continuous one is null. This in particular entails that $\nu_m$ has no atoms for any $m$.

\noindent Indeed, from a consequence of the de la Vall\'{e}e Poussin Theorem (see \cite{Sch}, Theorem F.6), one has the singular part of any positive measure $\m$ is supported by the set
$$
S_{\m}:=\{x\in \supp(\m); |\Im(\cg_m(x+i0))|=+\infty\}.
$$
From Proposition \ref{wig} the vacuum law of $G_1$ is absolutely continuous. Suppose this property  holds for any $G_1+\cdots +G_k$, $k<m$, i.e. $S_{\nu_k}=\emptyset$.
We prove $S_{\nu_m}$ is empty too. Indeed, suppose $S_{\nu_m}\neq \emptyset$. Then, if $x_0\in S_{\nu_m}$, since
$$
\limsup_{y\searrow 0}\big|\Im(\cf_m(x_0+iy))\big|=\limsup_{y\searrow 0}\bigg|-\frac{\Im(\cg_m(x_0+iy))}{|\cg_m(x_0+iy)|^2}\bigg|=0,
$$
from \eqref{cauchym2} and the triangular inequality, one has there exists $j=1,\ldots, m-1$ such that
$$
\big|\Im(\cg_j(x_0+i0))\big|=+\infty.
$$
Since  $\cg_j$ is analytic at any point outside the support of $\nu_j$, such a condition entails $x_0\in \supp(\nu_j)$. This  contradicts the induction assumption.
\end{proof}
In the following lines we give a recurrence formula for the right endpoint of the compact support for any $\nu_m$, and an evaluating formula to estimate its value without using the recurrence relation.

To this aim we use some properties of $\cf_1(z)$. We preliminary notice it maps $\bc^{+}\cup \br$ into $\bc^{+}\cup \br$ without the open unit half-disk. Indeed, it is known that $\cf_1$ maps $\bc^{+}$ into $\bc^{+}$, and if $z\in]-2,2[$, one finds $\cf_1(z)\in \bc^{+}$ with $|\cf_1(z)|=1$, whereas for $z\in\br$ with $|z|>2$, then $\cf_1(z)$ is real with $|\cf_1(z)|>1$. It remains to check any $w$ in the range of $\cf_1$ has modulus not less than $1$. After noticing that conditions above give $w\neq0$, take $z\in \bc^{+}\cup \br$ such that $\cf_1(z)=w$. This means that $z=w+\frac{1}{w}$. The above discussion suggests we can restrict to $z\in\mathbb{C}^+$. The last inequality in this case gives $\bigg(1-\frac{1}{|w|^2}\bigg)\Im w=\Im z$, which automatically entails $|w|^2\geq 1$.

As $\cf_1$ is one-to-one, we denote by $Z$ its composition inverse, which turns out to be conformal, i.e. it is the so-called \emph{Zhukovsky map}
$$
Z(w)=w+\frac{1}{w}.
$$
Since \eqref{cauchym2}, one has $Z$ is the inverse of $\cg_1$ too, and after denoting by $Z_m$ the $m$-fold composition of $Z$ with itself, it is easy to check
\begin{equation}
\label{zou}
Z_m(\cf_m(z))=z.
\end{equation}
\begin{thm}
\label{am+1thm}
Let $\pm a_m$ be the endpoints of the support of $\nu_m$. Then, for each $m\geq1$, one has $\supp(\nu_m)=[-a_m, a_m]$ and
\begin{equation}
\label{am+1}
a_{m+1}=a_m+\frac{1}{a_m}.
\end{equation}
\end{thm}
\begin{proof}
We first recall that any $\cf_m$ is holomorphic in $\mathbb{C^+}\cup (\br \setminus \supp(\nu_m))$, and prove that for any $m\geq2$,
\begin{equation}
\label{supp}
\supp(\nu_{m-1})\subseteq \supp(\nu_m).
\end{equation}
As the measures involved are all symmetric, it suffices to check that any positive $x\in\supp(\nu_{m-1})$ automatically belongs to $\supp(\nu_m)$. If this does not happen, then there exists a positive $x_0$ in $\supp(\nu_{m-1})$ where $\cf_m$ and then $\cg_m$ are holomorphic. This contradicts \eqref{recur}, as for any $z$ one has $\cg_m(z)=\cg_1(\cf_{m-1})(z)$, and $\cf_{m-1}$ is not analytic in $x_0$. As a consequence, it follows $a_{m-1}\leq a_m$. We now show that for each $m\geq 1$,
\begin{equation}
\label{Fmam}
Z_m(1)=a_m.
\end{equation}
In fact, $\cf_m=\cf_1\circ \cf_{m-1}$ and any $\cf_m$ is a real, increasing and continuous map if restricted to $[a_m,+\infty)$. This entails $z=\pm a_m$ satisfies $\cf_{m-1}(z)^2-4=0$. After recalling that $a_1=2$, one finds $\cf_{m-1}(a_m)=a_1$. As $\cf_1(a_1)=1$, one then has $\cf_m(a_m)=1$. This relation is equivalent to \eqref{Fmam}, as a consequence of \eqref{zou}. Thus
$$
a_{m+1}=Z_{m+1}(1)=Z(a_m)=a_m+\frac{1}{a_m}.
$$
Finally, we prove any real $x$ with modulus less than or equal to $a_m$ belongs to $\supp(\nu_m)$. Indeed, $\supp(\nu_1)=[-2, 2]$, and we further assume that $\supp(\nu_k)=[-a_k, a_k]$, $k<m$. Again it is enough to get $x$ positive, and by \eqref{supp} and \eqref{am+1}, we need to show that $x\in \supp(\nu_m)$ only when $a_{m-1}<x< a_m$. Since $1=\cf_{m-1}(a_{m-1})\leq\cf_{m-1}(x)\leq \cf_{m-1}(a_{m})=2$, one obtains $\cf_{m-1}(x)\in\supp(\nu_1)$. Therefore $\cf_m=\cf_1\circ \cf_{m-1}$ is not holomorphic in $x$.
\end{proof}
It is well known that rescaled sums of monotonically independent algebraic random variables with mean zero and variance 1 converge in the sense of moments to the arcsine law supported in $(-\sqrt{2},\sqrt{2})$ \cite{Mur}. It appears then natural to imagine that $\lim_{n}\frac{a_n}{\sqrt{n}}=\sqrt{2}$. This is in fact an immediate consequence of the following result, which gives an estimate for the magnitude of the supports of the vacuum distributions for sums of position operators in the weakly monotone Fock space.
\begin{thm}
\label{estim}
For any $m\geq 3$, the right endpoints $a_m$ of the support of the $m$-fold monotone convolution of the standard Wigner law, satisfy
\begin{equation}
\label{sqrt2nan}
\sqrt{m+\sqrt{m(m+1)}}\leq a_m\leq \sqrt{2m+\sqrt{2m}}.
\end{equation}
\end{thm}
\begin{proof}
Indeed, we first notice the Zhukovsky map $Z$, when restricted to the reals, is increasing in $[1,\infty)$. Since from \eqref{am+1} $a_3=\frac{29}{10}$, the inequalities in \eqref{sqrt2nan} hold when $m=3$. Suppose they are true for any $k\leq m$. For the right inequality, one easily verifies that for any $l\geq 1$
\begin{equation}
\label{2m}
\sqrt{l+2+\sqrt{l+2}}\geq \sqrt{l+\sqrt{l}}+\frac{1}{\sqrt{l+\sqrt{l}}}.
\end{equation}
The induction assumption and \eqref{am+1} ensure
\begin{equation*}
a_{m+1}\leq Z\bigg(\sqrt{2m+\sqrt{2m}}\bigg)\leq\sqrt{2(m+1)+\sqrt{2(m+1)}},
\end{equation*}
the last inequality coming from \eqref{2m}, when $l=2m$.\\
For the left inequality in \eqref{sqrt2nan}, after noticing that
$$
a_{m+1}=Z(a_m)\geq Z\bigg(\sqrt{m+\sqrt{m(m+1)}}\bigg)
$$
it is enough to check
\begin{equation*}
\label{left}
\sqrt{m+\sqrt{m(m+1)}}+ \frac{1}{\sqrt{m+\sqrt{m(m+1)}}}\geq \sqrt{m+1+\sqrt{(m+1)(m+2)}}.
\end{equation*}
The above inequality, after squaring and obvious cancelations, is equivalent to
$$
\sqrt{m(m+1)}+1+\frac{1}{m+\sqrt{m(m+1)}}\geq \sqrt{(m+1)(m+2)}.
$$
Now we observe that $1+\frac{1}{m+\sqrt{m(m+1)}}=\sqrt{(m+1)/m}$, which easily leads to the conclusion.
\end{proof}
A nice consequence of the estimate above is that the support of the dilated measures $\big(\frac{\nu_n}{\sqrt{n}}\big)_{n\geq 1}$ is bigger than the support of the standard arcsine law, and decreases more and more when $n$ is growing. More in detail, for any $n$, one has
\begin{equation}
\label{dec}
\frac{a_{m+1}}{a_m}\leq \sqrt{\frac{m+1}{m}}\,.
\end{equation}
In fact, one firstly notices that the left inequality in \eqref{sqrt2nan} holds for $m\leq 2$ too, and entails
$$
a_m^2\bigg(\frac{\sqrt{m+1}-\sqrt{m}}{\sqrt{m(m+1)}}\bigg)\geq \frac{1}{\sqrt{m+1}},
$$
which turns out to be equivalent to
$$
\frac{a_m}{\sqrt{m}}\geq \frac{1}{\sqrt{m+1}}\bigg(a_m+\frac{1}{a_m}\bigg).
$$
This last inequality gives \eqref{dec} after using \eqref{am+1}.

\section{appendix}

In the next lines we present some computations, obtained in collaboration with F. Lehner, related to the features developed in the previous sections.

\noindent We first present some examples of sequences of moments for the sum of $m$ position operators, the recursive formula being \eqref{dn^m}.
\begin{eqnarray*}
d_n^{(2)}&=&1, 2, 7, 29, 131, 625, 3099, 15818, 82595, \ldots \\
d_n^{(3)}&=&1, 3, 15, 87, 544, 3566, 24165, 167904, \ldots \\
d_n^{(4)}&=&1, 4, 26, 194, 1551, 12944, 111313, 979009, \ldots\\
d_n^{(5)}&=&1, 5, 40, 365, 3555, 36045, 375797, 4000226, 43279506, \ldots \\
d_n^{(6)}&=&1, 6, 57, 615, 7064, 84307, 1033089, 12909546,  163799094, \ldots \\
d_n^{(7)}&=&1, 7, 77, 959, 12691, 174265, 2454221, 35215061, 512675782,  \ldots\\
d_n^{(8)}&=&1, 8, 100, 1412, 21154, 328496, 5227522, 84698378, 1391557207,  \ldots\\
d_n^{(9)}&=&1, 9, 126, 1989, 33276, 576564, 10230750, 184733379, 3380878107,  \ldots\\
d_n^{(10)}&=&1, 10, 155, 2705, 49985, 955965, 18713619, 372615462, 7517051642,  \ldots\\
\end{eqnarray*}

\noindent In Proposition \ref{momentspoly} we showed that for any $n\geq0$, $d_n^{(m)}$ is a polynomial of degree $n$ on $m$. Below they are listed until $n=7$.
\begin{equation*}
\begin{split}
d_0^{(m)}&\equiv 1\\
d_1^{(m)} & = m \\
d_2^{(m)} & = \frac{3m^2+m}{2} \\
d_3^{(m)} & = \frac{m(5m^2+4m+1)}{2}\\
d_4^{(m)}&=\frac{35}{8}m^4+\frac{71}{12}m^3+\frac{25}{8}m^2+\frac{7}{12}m \\
d_5^{(m)}&=\frac{63}{8}m^5+\frac{31}{2}m^4+\frac{311}{24}m^3+5m^2+\frac{2}{3}m\\
d_6^{(m)}&=\frac{231}{16}m^6+\frac{3043}{80}m^5+\frac{2135}{48}m^4+\frac{429}{16}m^3+\frac{91}{12}m^2+\frac{13}{20}m\\
d_7^{(m)}&=\frac{429}{16}m^7+\frac{2689}{30}m^6+\frac{4099}{30}m^5+\frac{685}{6}m^4+\frac{2453}{48}m^3+\frac{51}{5}m^2+\frac{9}{20}m.\\
\end{split}
\end{equation*}

\noindent The first terms of the monotone cumulants $(r_n)_{n\geq 1}$ for the standard semicircle law are
$$
0, 1, 0, \frac{1}{2},0,\frac{1}{2},0,\frac{7}{12},0,\frac{2}{3},0, \frac{13}{20},0,\frac{9}{20},0,\frac{71}{280},0,\frac{121}{40},0,\frac{19}{7}.
$$

\noindent Finally, we find the first orthogonal polynomials for the distribution with density $\gpg(x)$ appeared in Theorem \ref{dens}.
\begin{align*}
&p_1(x)=x \\
&p_2(x)=x^2-x \\
&p_3(x)=x^3- \frac{7}{2}x \\
&p_4(x)=x^4-5x^2+3 \\
&p_5(x)=x^5-\frac{59}{9}x^3+\frac{76}{9}x \\
&p_6(x)=x^6-\frac{57}{7}x^4+\frac{344}{21}x^2-\frac{100}{21}\\
&p_7(x)=x^7-\frac{243}{25}x^5+\frac{668}{25}x^3-\frac{452}{25}x \\
&p_8(x)=x^8-\frac{4149}{368}x^6+\frac{14491}{368}x^4-\frac{2003}{46}x^2+\frac{681}{92} \\
&p_9(x)=x^9-\frac{2911}{227}x^7+\frac{49429}{908}x^5-\frac{77127}{908}x^3+\frac{8039}{227}x \\
&p_{10}(x)=x^{10}-\frac{74473}{5177}x^8+\frac{372967}{5177}x^6-\frac{758082}{5177}x^4+\frac{535355}{5177}x^2-\frac{59841}{5177}.
\end{align*}

\bigskip

\noindent \textbf{Acknowledgements.}
The first and the second named authors acknowledge the support of the italian INDAM-GNAMPA and Fondi di Ateneo Universit\`a di Bari ``Probabilit\`a Quantistica e Applicazioni''.\\
J. Wysoczanski acknowledges the support by the Polish National Science Center Grant No. 2016/21/B/ST1/00628.\\
The authors also express deep gratitude toward profs. F. Lehner and W. M{\l}otkowski for useful discussions, and thank an anonymous referee whose nice comments improved the presentation of the paper.

\end{document}